\documentclass[11pt,french,english]{amsart}
\usepackage[T1]{fontenc}
\usepackage[latin1]{inputenc}
\usepackage{amsthm}
\usepackage{amssymb}

\makeatletter
\numberwithin{equation}{section}
\numberwithin{figure}{section}
\theoremstyle{plain}
\newtheorem{thm}{\protect\theoremname}
  \theoremstyle{plain}
  \newtheorem{cor}[thm]{\protect\corollaryname}
  \theoremstyle{remark}
  \newtheorem{rem}[thm]{\protect\remarkname}
  \theoremstyle{plain}
  \newtheorem{lem}[thm]{\protect\lemmaname}
  \theoremstyle{plain}
  \newtheorem{prop}[thm]{\protect\propositionname}

\makeatother

\makeatother

\usepackage{babel}
\makeatletter
\addto\extrasfrench{%
   \providecommand{\fg}{\ifdim\lastskip>\z@\unskip\fi~\frqq}%
}

\makeatother
  \addto\captionsenglish{\renewcommand{\corollaryname}{Corollary}}
  \addto\captionsenglish{\renewcommand{\lemmaname}{Lemma}}
  \addto\captionsenglish{\renewcommand{\propositionname}{Proposition}}
  \addto\captionsenglish{\renewcommand{\remarkname}{Remark}}
  \addto\captionsenglish{\renewcommand{\theoremname}{Theorem}}
  \addto\captionsfrench{\renewcommand{\corollaryname}{Corollaire}}
  \addto\captionsfrench{\renewcommand{\lemmaname}{Lemme}}
  \addto\captionsfrench{\renewcommand{\propositionname}{Proposition}}
  \addto\captionsfrench{\renewcommand{\remarkname}{Remarque}}
  \addto\captionsfrench{\renewcommand{\theoremname}{Théorème}}
  \providecommand{\corollaryname}{Corollary}
  \providecommand{\lemmaname}{Lemma}
  \providecommand{\propositionname}{Proposition}
  \providecommand{\remarkname}{Remark}
\providecommand{\theoremname}{Theorem}

\begin{document}

\newcommand{\Gal}{{\rm Gal}} 
\newcommand{\Pic}{{\rm Pic}}

\addtolength{\textwidth}{12mm}
\addtolength{\hoffset}{-7mm}
\addtolength{\textheight}{25mm}
\addtolength{\voffset}{-7mm}

%\subjclass{Primary: 14J29 ; Secondary: 14G10, 14G15}

\global\long\global\long\def\Alb{{\rm Alb}}
 \global\long\global\long\def\Jac{{\rm Jac}}
 \global\long\global\long\def\Hom{{\rm Hom}}
 \global\long\global\long\def\End{{\rm End}}
 \global\long\global\long\def\aut{{\rm Aut}}
 \global\long\global\long\def\NS{{\rm NS}}
 \global\long\global\long\def\SSm{{\rm S}}
 \global\long\global\long\def\psl{{\rm PSL}}
 \global\long\global\long\def\CC{\mathbb{C}}
 \global\long\global\long\def\BB{\mathbb{B}}
 \global\long\global\long\def\PP{\mathbb{P}}
 \global\long\global\long\def\QQ{\mathbb{Q}}
 \global\long\global\long\def\RR{\mathbb{R}}
 \global\long\global\long\def\FF{\mathbb{F}}
 \global\long\global\long\def\DD{\mathbb{D}}
 \global\long\global\long\def\NN{\mathbb{N}}
 \global\long\global\long\def\ZZ{\mathbb{Z}}
 \global\long\global\long\def\HH{\mathbb{H}}
 \global\long\global\long\def\gal{{\rm Gal}}
 \global\long\global\long\def\OO{\mathcal{O}}
 \global\long\global\long\def\pP{\mathfrak{p}}
 \global\long\global\long\def\pPP{\mathfrak{P}}
 \global\long\global\long\def\qQ{\mathfrak{q}}
 \global\long\global\long\def\mm{\mathcal{M}}
 \global\long\global\long\def\aaa{\mathfrak{a}}

\title{The zeta functions of the Fano surfaces of cubic threefolds}

\author{Xavier Roulleau}
\begin{abstract}
We give an algorithm to compute the zeta function of the Fano surface
of lines of a smooth cubic threefold $F\hookrightarrow\mathbb{P}^{4}$
defined over a finite field. We obtain some examples of Fano surfaces
with supersingular reduction. 
\end{abstract}
\maketitle

\section{Introduction.}

Let $k=\mathbb{F}_{q}$ be a finite field of characteristic $p\not=2$
and let $F\hookrightarrow\mathbb{P}_{/k}^{4}$ be a smooth cubic hypersurface
defined over $k$. Through a generic point of $\bar{F}$ pass 6 lines
of $\bar{F}$. The variety that parametrizes the lines on $F$ is
a smooth projective surface defined over $k$ called the Fano surface
of lines of $F$. This surface $S$ is minimal of general type.

In \cite{RoulleauTate}, we prove that the Tate conjecture holds for
$S$. Recall that the zeta function $Z(X,T)$ of a smooth variety
$X_{/k}$ encodes the number of rational points of $X$ over the extensions
$k_{r}=\mathbb{F}_{q^{r}}$. The Tate conjecture for $X$ predicts
that the Picard number $\rho_{X}$ of $X_{/k}$ equals to the order
of pole of the zeta function $Z(X,q^{-s})$ at $s=1$. 

Knowing that the Tate conjecture holds for a Fano surface $S$, it
is then natural to study its consequences and applications : this
is the aim of this paper. Using the results of Bombieri and Swinnerton-Dyer
\cite{Bombieri}, we give an algorithm that compute the zeta function
$Z(S,T)$ for a Fano surface $S$ defined over $k$ containing a $k$-rational
point. We are therefore able to obtain the Picard number $\rho_{S}$,
which is an important invariant but usually difficult to compute.

The intermediate Jacobian $J(F)$ of the cubic $F$ is a certain $5$-dimensional
Abelian variety canonically associated to $F$. This variety $J(F)$
is isomorphic over $k$ to the Albanese variety of $S$, and it is
also a Prym variety $Pr(C/\Gamma)$ associated to a degree $2$ cover
$C\to\Gamma$ of a certain plane quintic curve $\Gamma$. As we will
see, it is equivalent to compute the zeta functions $Z(F,T),\, Z(S,T)$
or $Z(J(F),T)$.

In \cite{Kedlaya}, Kedlaya gives an algorithm to compute the zeta
function of cubic threefolds $F$ that are triple abelian cover of
$\mathbb{P}^{3}$ branched over a smooth cubic surface. This algorithm
runs by computing the number of points on the cubic hypersurface $F\hookrightarrow\mathbb{P}^{4}$
and uses the order $3$ symmetries of $F$ in order to reduce the
computations. The algorithm we give (implemented in Sage, see \cite{RoulleauWorksheet})
runs for any cubics containing a $k$-rational line. To our knowledge,
it seems also to be the first algorithm proposed in order to compute
the zeta function of a Prym variety. 

We then study some interesting examples of cubic threefolds and compute
the zeta functions of their Fano surfaces over fields of characteristic
$5$ and $7$. We study also the Fano surface $S_{/\mathbb{Q}}$ of
the Klein cubic threefold : 
\[
F_{/\mathbb{Q}}=\{x_{1}^{2}x_{2}+x_{2}^{2}x_{3}+x_{3}^{2}x_{4}+x_{4}^{2}x_{5}+x_{5}^{2}x_{1}=0\}.
\]
 The surface $S$ has good reduction $S_{p}$ at every prime $p\not=11$. 
\begin{cor}
Let be $p\not=2,\,11$. If $11$ is a square modulo $p$, then $S_{p}$
has geometric Picard number $25$, otherwise $S_{p}$ is supersingular,
i.e. its geometric Picard number equals $45=b_{2}$. 
\end{cor}
To the knowledge of the author, there are only a few cases where the
zeta function of surfaces have been computed : apart from the cases
of Abelian surfaces and a product of two curves, the zeta function
has been computed for some examples of $K3$, some elliptic surfaces,
for Fermat and some Delsarte hypersurfaces (see \cite{Schuett} and
the references therein).

\textbf{Acknowledgements.} The author gratefully thanks Bert van Geemen
for his suggestion to consider the conic bundle structure for computing
the zeta function of cubics, Kiran Kedlaya for emails exchanges and
Olivier Debarre for its help in the proof of Corollary \ref{cor:Picard number 5}.
This paper has been supported under FCT grant SFRH/BPD/72719/2010
and project Géometria Algébrica PTDC/MAT/099275/2008.

\section{The zeta function of a cubic threefold and its Fano surface.}

\subsection{Notations, hypothesis.}

Let $F$ be a smooth cubic hypersurface defined over a field $k$
of characteristic not $2$ and let $S$ be its Fano surface. The surface
$S$ is a smooth geometrically connected variety defined over $k$
\cite[Thm 1.16 i and (1.12)]{Altman}. We will suppose that $S$ has
a rational point $s_{o}$ i.e. that $F$ contains a line defined over
$k$. The Albanese variety $A$ of $S$ is $5$ dimensional and is
defined over $k$ \cite[Lem. 3.1]{Achter}. Let $\vartheta:S\to A$
be the Albanese map such that $\vartheta s_{o}=0$. Let $\Theta$
be the (reduced) image of $S\times S$ under the map $(s_{1},s_{2})\to\vartheta s_{1}-\vartheta s_{2}$.
The variety $\Theta$ is a divisor on $A$ defined over $k$ and $(A,\Theta)$
is a principally polarized abelian variety \cite[Prop. 5]{Beauville}.
In this paper, we call the pair $J(F)=(A,\Theta)$ the intermediate
Jacobian of $F$.

\subsection{Zeta functions, Weil polynomials}

Let $X$ be a smooth projective $n$-dimensional variety over a finite
field $k=\mathbb{F}_{q}$ such that $\bar{X}$ (the variety over an
algebraic closure $\bar{k}$ of $k$) is smooth. Let $N_{r}$ be the
number of rational points of $\bar{X}$ over $\mathbb{F}_{q^{r}}$.
The zeta function of $X$ is defined by $Z(X,T)=exp(\sum_{r\geq0}N_{r}\frac{T^{r}}{r})$
and can be written: 
\[
Z(X,T)=\sum_{i=0}^{i=2n}P_{i}(X,T)^{(-1)^{i+1}}
\]
 where $P_{i}(X,T)$ is the Weil polynomial with integer coefficients:
\[
P_{i}(X,T)=\det(1-\pi^{*}|H^{i}(\bar{X},\mathbb{Q}_{\ell})),
\]
 for $\pi$ the Frobenius endomorphism of $\bar{X}$. Suppose that
$X$ is a surface. Let $Br(X)$ be the Brauer group of $X$ and let
$NS_{k}(X)$ be its Néron-Severi group. 
\begin{thm}
(\cite{Milne}, Artin-Tate Conjecture). Suppose that the Tate conjecture
is satisfied by $X$ and the characteristic of $k$ is at least $3$.
The Brauer group of $X$ is finite and:
\[
\lim_{s\to1}\frac{P_{2}(X,q^{-s})}{(1-q^{1-s})^{\rho_{X}}}=\frac{(-1)^{\rho_{X}-1}|Br(X)|Disc(NS_{k}(X))}{q^{\alpha(X)}|NS_{k}(X)_{tor}|^{2}}
\]
 where $\alpha(X)=\chi(X,\mathcal{O}_{X})-1+\dim PicVar(X)$, $NS_{k}(X)_{tor}$
is the torsion sub-group of $NS_{k}(X)$, $\rho_{X}$ and $Disc(NS_{k}(X))$
are respectively the rank and the absolute value of the discriminant
of $NS_{k}(X)/NS_{k}(X)_{tor}$. 
\end{thm}
Note that the order of a finite Brauer group is always a square. Let
$S$ be the Fano surface of a smooth cubic threefold $F\hookrightarrow\mathbb{P}^{4}$.
The Picard variety $PicVar(S)$ of $S$ is reduced of dimension $5$
\cite[Lem. 1.1]{Tyurin1}, and $\chi=6$, therefore $\alpha(S)=10$. 

As $H^{2}(\bar{S},\mathbb{Q}_{\ell})\simeq\wedge^{2}H^{1}(\bar{A},\mathbb{Q}_{\ell})\simeq H^{2}(\bar{A},\mathbb{Q}_{\ell})$,
the zeta function of $S$ can be computed if we know the zeta function
of $A$ ; in fact we just need to know the action of the Frobenius
on $H^{1}(\bar{A},\mathbb{Q}_{\ell})$.
\begin{thm}
Suppose characteristic $\not=2$. The étale cohomology groups $H^{3}(\bar{F},\mathbb{Q}_{\ell})$
and $H^{1}(\bar{A},\mathbb{Q}_{\ell}(1))$ are isomorphic as Galois
modules.\end{thm}
\begin{proof}
By \cite[(14) and Thm 3]{Murre}, we have $H^{3}(\bar{F},\mathbb{Q}_{\ell})\simeq T_{\ell}(A)\otimes\mathbb{Q}_{\ell}(1)$
where $T_{\ell}(A)$ is the Tate module of $A$. But $T_{\ell}(A)\otimes\mathbb{Q}_{\ell}$
is isomorphic to $H^{1}(\bar{A},\mathbb{Q}_{\ell})$.\end{proof}
\begin{rem}
It would be interesting to know a bound on the order of the torsion
subgroup $NS(S)_{tor}$. In the examples we give in Section 6, we
compute the values of $\lim_{s\to1}\frac{P_{2}(S,q^{-s})}{(1-q^{1-s})^{\rho_{S}}}$.
For all the limits we computed, we get results on the form $\frac{N}{q^{10}},\, N\in\mathbb{Z}$,
giving the idea that the group $NS(S)_{tor}$ should be trivial. 
\end{rem}

\subsection{Zeta function of a smooth cubic threefold.}

The upper half part of the Hodge diamond of a smooth cubic threefold
$F$ over $\mathbb{C}$ is:
\[
\begin{array}{ccccccc}
 &  &  & 1\\
 &  & 0 &  & 0\\
 & 0 &  & 1 &  & 0\\
0 &  & 5 &  & 5 &  & 0
\end{array}
\]
 Thus $P_{0}(F,T)=1-T$, $P_{1}(F,T)=1=P_{5}(F,T)$, $P_{2}(F,T)=1-qT$.
The polynomials $P_{4}$ and $P_{6}$ are computed by Poincaré duality:
$P_{4}=1-q^{2}T,\, P_{6}(F,T)=1-q^{3}T$. 
\begin{cor}
We have:
\[
Z(F,T)=\frac{P_{3}(F,T)}{(1-T)(1-qT)(1-q^{2}T)(1-q^{3}T)},
\]
 and
\[
N_{r}(F)=(1+q^{r}+q^{2r}+q^{3r})-q^{r}Tr(\pi^{r}|H^{1}(\bar{A},\mathbb{Q}_{\ell})),
\]
 where $\pi$ is the Frobenius endomorphism of $\bar{A}$ and $Tr$
denotes the trace. 
\end{cor}
The roots of $P_{i}(F,T)$ have absolute value $q^{-i/2}$. We have:
\[
P_{1}(S,T)=P_{1}(A,T)=P_{3}(F,\frac{T}{q}).
\]
 Since by a computer it is possible to compute the number of points
of $F$, it is possible in theory to get the zeta function of $F,A$
and $S$. Kedlaya's algorithm computes the number of points of cubics
$F$ that are cyclic triple cover of $\mathbb{P}^{3}$. Alternatively,
we can also compute the number of points of $A$, which is a Prym
variety of a degree $2$ étale cover of a certain plane quintic ;
this is the algorithm we have implemented.

\subsection{Zeta function and Picard number of a Fano surface.}

Let $S$ be the Fano surface of a smooth cubic $F\hookrightarrow\mathbb{P}^{4}$
over $\mathbb{F}_{q}$. We have
\[
Z(S,T)=\frac{P_{1}(S,T)P_{3}(S,T)}{(1-T)P_{2}(S,T)(1-q^{2}T)}.
\]
 We can write $P_{1}=\prod_{i=1}^{i=10}(1-\omega_{i}T)\in\mathbb{Z}[T]$
with $|\omega_{i}|=q^{1/2}$. As 
\[
H^{2}(\bar{S},\mathbb{Q}_{\ell})=\wedge^{2}H^{1}(\bar{S},\mathbb{Q}_{\ell})
\]
 we get $P_{2}(S,T)=\prod_{1\leq i<j\leq10}(1-\omega_{i}\omega_{j}T)$.
Moreover $P_{3}(S,T)=q^{15}T^{10}P_{1}(S,\frac{1}{q^{2}T}).$ Thus
\begin{equation}
Z(S,T)=\frac{\prod_{i=1}^{i=10}(1-\omega_{i}T)\prod_{i=1}^{i=10}(1-\frac{q^{2}}{\omega_{i}}T)}{(1-T)(1-q^{2}T)\prod_{1\leq i<j\leq10}(1-\omega_{i}\omega_{j}T)}.\label{Zeta Function formula}
\end{equation}
 The order of the pole at $1$ of $Z(S,q^{-s})$ equals the number
of elements of the set $\{(i,j)/1\leq i<j\leq10\mbox{ and }\omega_{i}\omega_{j}=q\}$
; it is the multiplicity of the root $1/q$ in $P_{2}(S,T)=P_{2}(A,T)$.
The Tate conjecture for Fano surfaces proved in \cite{RoulleauTate}
predicts that the Picard number $\rho_{S}$ of $S$ is equal to the
order of pole of the zeta function $Z(S,q^{-s})$ at $s=1$, thus
by \ref{Zeta Function formula}, we get: 
\begin{cor}
\label{cor:Picard number 5}The Fano surface has Picard number at
least $5$. \end{cor}
\begin{proof}
This follows from the more general fact that over a finite field,
the rank of the Néron-Severi group of an Abelian variety $A$ is always
larger or equal to $\dim A=g$. For the proof, one may assume that
$A$ is simple. Let $\omega_{1},...,\omega_{2g}$ be the inverse of
the roots of the characteristic polynomial $P_{1}$. We must prove
that at least $g$ among the products $\omega_{i}\omega_{j}$, for
$i<j$, are equal to $q$. Since $q/\omega_{i}$ is also a root, the
difficulty occurs only when there are real roots, ie roots equal to
$\pm q^{1/2}$. Since $A$ is simple, the Honda-Tate Theorem (\cite[Th. 4.7.2]{MZ})
implies that if $q^{1/2}\notin\ZZ$, then $g=2$, $P_{1}(T)=(T^{2}-q)^{2}$
and if $q'=q^{1/2}\in\ZZ$, then $P_{1}(T)=(T\pm q')^{2}$, $g=1$.
In both cases the multiplicities of the roots are $2$.
\end{proof}
Let $N_{r}(X)$ be the number of rational points of a variety $X$
over $\mathbb{F}_{q^{r}}$. For each Weil polynomial $P_{i}(X,T)$,
let $\omega_{i,1},\dots,\omega_{i,b_{i}}$ be the reciprocal roots
of $P_{i}$. The formula expressing the zeta function gives:
\[
N_{r}(X)=\sum_{i=0}^{i=2\dim X}(-1)^{i}\sum_{j=1}^{j=b_{i}}\omega_{i,j}^{r}.
\]
 We will use in section \ref{section examples} this formula for computing
the numbers $N_{1}(S),N_{2}(S)$ of the Fano surface $S$ from the
knowledge of the reciprocal roots $\omega_{1,1},\dots,\omega_{1,10}$
of $P_{1}(S,T)$.

\section{The intermediate Jacobian as a Prym variety.}

Let $F\hookrightarrow\mathbb{P}^{4}$ be a smooth cubic threefold
defined over the finite field $k=\mathbb{F}_{q}$ of characteristic
not $2$. We will suppose that $F$ contains a line $L$ defined over
$k$, or equivalently that the Fano surface contains a $k$-rational
point. 

The aim of this section is to give an algorithm that compute the zeta
function of the cubic $F$, the Fano surface $S$ and the Albanese
variety $A$ by using the conic bundle structure on an appropriate
blow-up of $F$. 

The algorithm will use two auxiliary curves : $C_{L}$ the incidence
curve parametrizing the lines on $F$ meeting $L$ and $\Gamma_{L}$
the curve parametrizing planes $Y$ containing $L$ and such that
the degree $3$ plane curve $Y\cap F$ is the union of three lines.

Let $x_{1},\dots,x_{5}$ be coordinates in $\mathbb{P}^{4}$ and let
$L\hookrightarrow F$ be a $k$-rational line on $F$. We can suppose
that $L=\{x_{1}=x_{2}=x_{3}=0\}$ and $F=\{F_{eq}=0\}$ with 
\[
F_{eq}=\ell_{1}x_{4}^{2}+2\ell_{2}x_{4}x_{5}+\ell_{3}x_{5}^{2}+2q_{1}x_{4}+2q_{2}x_{5}+f,
\]
 where $\ell_{1},\ell_{2},\ell_{3}$ are linear, $q_{1},q_{2}$ are
quadratic and $f$ is a cubic homogenous forms in the variables $x_{1},x_{2},x_{3}$.
Let $X\hookrightarrow\mathbb{P}^{4}$ be the plane $x_{4}=x_{5}=0$.
Any plane $Y\hookrightarrow\mathbb{P}^{4}$ containing the line $L$
meets $X$ into a unique point. Thus the plane $X$ parametrizes planes
in $\mathbb{P}^{4}$ containing $L$. For a $\bar{k}$-point $x=(x_{1}:x_{2}:x_{3}:0:0)$
of $X=\mathbb{P}^{2}$, we denote by $Y_{x}$ the unique plane containing
$x$. The intersection of $F$ and $Y_{x}$ has equation $F_{eq}(y_{1}x_{1}:y_{1}x_{2}:y_{1}x_{3}:y_{2}:y_{3})=0$,
i.e.: 
\[
y_{1}(\ell_{1}y_{2}^{2}+2\ell_{2}y_{2}y_{3}+\ell_{3}y_{3}^{2}+2q_{1}y_{1}y_{2}+2q_{2}y_{1}y_{3}+y_{1}^{2}f)=0
\]
 in the plane $Y_{x}$ with coordinates $y_{1},y_{2},y_{3}$. Therefore,
the intersection $Y\cap F$ equals $L+Q_{x}$ where 
\[
Q_{x}=\{\ell_{1}y_{2}^{2}+2\ell_{2}y_{2}y_{3}+\ell_{3}y_{3}^{2}+2q_{1}y_{1}y_{2}+2q_{2}y_{1}y_{3}+y_{1}^{2}f=0\}.
\]
 The conic $Q_{x}$ is irreducible over $\bar{k}$ if and only if
it is smooth. The scheme parametrizing the planes $Y_{x}$ such that
$Q_{x}$ decomposes (over an extension) as $Q_{x}=L_{1}+L_{2}$ with
$L_{1},L_{2}$ two lines, is a plane reduced quintic curve $\Gamma_{L}\hookrightarrow X=\mathbb{P}^{2}(x_{1},x_{2},x_{3})$
defined over $k$, of equation $\det M_{\Theta}=0$, where: 
\[
M_{\Theta}=\left(\begin{array}{ccc}
\ell_{1} & \ell_{2} & q_{1}\\
\ell_{2} & \ell_{3} & q_{2}\\
q_{1} & q_{2} & f
\end{array}\right).
\]
 Let $C_{L}$ be the incidence scheme parametrizing the lines that
meet $L$. 
\begin{lem}
(\cite[Lem. 2, 4]{Bombieri}, \cite[Pro. (1.25)]{MurreCompo1}). The
plane quintic curve $\Gamma_{L}$ has at most nodal singularities.
A points $x$ on $\Gamma_{L}$ is singular if and only if $Y_{x}\cap F=L+2L'$,
where $L'$ is a line. The curve $\Gamma_{L}$ is smooth for $L$
generic.\\
 The scheme $C_{L}$ is a reduced one dimensional subscheme of the
Fano surface $S$. There is a natural degree $2$ map $\mu:C_{L}\to\Gamma_{L}$
that is ramified precisely over the singularities of $\Gamma$. The
point $p$ of $C_{L}$ is smooth if and only if $\mu(p)$ is smooth.
\\
If $\Gamma_{L}$ is smooth, the Albanese variety of $S$ is isomorphic
to the Prym variety $Pr(C_{L}/\Gamma_{L})$.
\end{lem}
Let $\tilde{F}\to F$ be the blow-up of $F$ along a line $L$ on
$F$. The variety $\tilde{F}$ has a natural structure of a conic
bundle:
\[
\tilde{F}\to\mathbb{P}^{2}=X
\]
where $X$ parametrizes the planes containing the line $L$. The fiber
over the point $x$ is (isomorphic to) the quadric $Q_{x}$ such that
$F\cap Y_{x}=Q+L$. For $x$ defined over $k$, the number of $k$-points
of $Q_{x}$ is:\\
 i) $q+1$ if $Q_{x}$ is geometrically irreducible i.e. $x\not\in\Gamma_{L}(k)$
or if $x\in\Gamma_{L}(k)$ is a singular point of $\Gamma$, in that
case $Q_{x}=2L'$ with $L'$ a $k$-rational line on $F$.\\
 ii) $2q+1$ if $Q_{x}=L_{1}+L_{2}$ with $L_{1},L_{2}$ two different
lines defined over $k$, i.e. the $2$ points in $C_{L}$ over $x\in\Gamma_{L}(k)$
are in $C_{L}(k)$.\\
 iii) $1$ if the conic $Q_{x}$ degenerates to a union of two line
over a (degree $2$) extension of $k$, i.e. if the $2$ points in
$C_{L}$ over $x\in\Gamma_{L}(k)$ are not in $C_{L}(k)$.

Let $N_{r}(X)$ denotes the number of rational points of a variety
$X$ over the degree $r$ extension $k_{r}=\mathbb{F}_{q^{r}}$ of
$k$. The following Proposition is \cite[Formula (18)]{Bombieri},
however we reproduce the proof here because it explains the algorithm
computing $N_{r}(F)$ we describe below. 
\begin{prop}
\label{proposition number of points Bombieri}We have
\[
N_{r}(F)=q^{3r}+q^{2r}+q^{r}+1+q^{r}(N_{r}(C_{L})-N_{r}(\Gamma_{L})).
\]
\end{prop}
\begin{proof}
Let $a$ be the number of $\mathbb{F}_{q}$-rational singularities
of $\Gamma_{L}$ (and $C_{L}$). Taking care of the three above possibilities
i), ii) and iii), we get:
\[
\begin{array}{cc}
N_{1}(\tilde{F})= & (q+1)(N_{1}(\mathbb{P}^{2})-N_{1}(\Gamma_{L})+a)+(2q+1)\frac{1}{2}(N_{1}(C_{L})-a)\\
 & +(N_{1}(\Gamma_{L})-a-\frac{1}{2}(N_{1}(C_{L})-a))
\end{array}
\]
 thus
\[
N_{1}(\tilde{F})=q^{3}+2q^{2}+2q+1+q(N_{1}(C_{L})-N_{1}(\Gamma_{L})).
\]
 As each point on $L\hookrightarrow F$ is replaced by a $\mathbb{P}^{1}$
on $\tilde{F}$, we have moreover:
\[
N_{1}(\tilde{F})=N_{1}(F)-(q+1)+(q+1)^{2},
\]
 thus $N_{1}(F)=q^{3}+q^{2}+q+1+q(N_{1}(C_{L})-N_{1}(\Gamma_{L})).$
\end{proof}
Let $\mu:C_{L}\to\Gamma_{L}$ be the degree $2$ map and let $x$
be a $k$-rational smooth point on $\Gamma$. In order to compute
the numbers $N_{r}(C_{L})-N_{r}(\Gamma_{L})$, we need to understand
when the two $\bar{k}$-rational points in $\mu^{-1}x$ are $k$-rational.
For $1\leq i\leq3$, let $\delta_{i}\in H^{0}(\Gamma_{L},\mathcal{O}(a)),\, a=2\mbox{ or }4$
be the $(i,i)$-minor of the matrix $M_{\Theta}$. 
\begin{prop}
\label{pro:For-every-smooth delta}Let $x$ be smooth point of $\Gamma$.
There exists an integer $1\leq i=i(x)\leq3$ such that $\delta_{i}(x)\not=0$.
The curve $C_{L}$ has two rational points over $x\in\Gamma_{L}(k)$
if and only if $-\delta_{i}(x)\in(k^{*})^{2}.$\end{prop}
\begin{proof}
Let $x$ a $k$-rational smooth point of $\Gamma$ and let $Q=Q_{x}$
such that $F\cap Y_{x}=L+Q$. The line $L=\{y_{1}=0\}$ meets $Q_{x}\hookrightarrow Y_{x}$
in the points such that: 
\[
y_{1}=\ell_{1}y_{2}^{2}+2\ell_{2}y_{2}y_{3}+\ell_{3}y_{3}^{2}=0.
\]
Therefore, if $-\delta_{3}(x)=\ell_{2}^{2}(x)-\ell_{1}(x)\ell_{2}(x)$
is nonzero, the curve $C_{L}$ has two rational points over $x\in\Gamma_{L}(k)$
if and only if $-\delta_{3}(x)\in(k^{*})^{2}$. \\
 For $\delta_{3}(x)=0$, we only sketch the proof ; see also \cite[Lemme 1.6]{Beauville2}
and its proof. In that case, we have $Y_{x}\cdot F=L+Q_{x}=L+L_{1}+L_{2}$
with $L_{1},L_{2}$ defined over $\bar{k}$ and meeting in a $k$-rational
point $p$ of $L$. It is possible to explicitly compute a model in
$\mathbb{P}^{3}$ of (the degree $6$) scheme $\mathcal{X}_{p}$ of
lines in the cubic $F$ going through $p$. By knowing $\mathcal{X}_{p}$,
we can determine whether the two points on $\mathcal{X}_{p}$ corresponding
to the lines $L_{1},L_{2}$ are $k$-rational or not and this is so
if and only if $-\delta_{1}$ is a nonzero square or $-\delta_{2}$
is a nonzero square. 
\end{proof}
Let us describe the algorithm for the computation of the numbers $N_{r}(F)$
and $N_{r}(C_{L})-N_{r}(\Gamma_{L})$.

The input data is a cubic threefold $F$ over $\mathbb{F}_{q}$ containing
a $\mathbb{F}_{q}$-rational line $L$. To this data is associated
the matrix $M_{\Theta}$ defined above whose determinant is the equation
of the quintic $\Gamma_{L}\hookrightarrow X=\mathbb{P}^{2}$ (maybe
singular). Then we compute $N_{r}=N_{r}(C_{L})-N_{r}(\Gamma_{L})$
as follows:

Initiate $N_{r}:=0$. For $x\in\mathbb{P}^{2}(\mathbb{F}_{q^{r}})$,
if $\det(M_{\Theta})(x)=0$ then if $-\delta_{3}(x)\in(\mathbb{F}_{q^{r}}^{*})^{2}$
then $N_{r}:=N_{r}+1$, else if $-\delta_{3}(x)\not=0$, then $N_{r}:=N_{r}-1$,
otherwise if $-\delta_{1}(x)\in(\mathbb{F}_{q^{r}}^{*})^{2}$ then
$N_{r}:=N_{r}+1$ else if $-\delta_{1}(x)\not=0$, then $N_{r}:=N_{r}-1$,
otherwise if $-\delta_{2}(x)\in(\mathbb{F}_{q^{r}}^{*})^{2}$ then
$N_{r}:=N_{r}+1$ else if $-\delta_{2}(x)\not=0$ then $N_{r}:=N_{r}-1$
end if, end for.

The output $N_{r}$ equals $N_{r}(C_{L})-N_{r}(\Gamma_{L})$. Remark
that the $-\delta_{i}$ are transition functions of an invertible
sheaf $\mathcal{L}$ on $\Gamma_{L}$ such that $\mathcal{L}^{\otimes2}=\omega_{\Gamma}$.
The data of $\mathcal{L}$ corresponds to the degree $2$ cover $\mu:C_{L}\to\Gamma_{L}$
and a point on $\Gamma_{L}$ is singular if and only if $\forall1\leq i\leq3,\,\delta_{i}(x)=0$.

The knowledge of $N_{1},\dots,N_{5}$ is enough to get the $5$ first
coefficients of the degree $10$ polynomial $P_{1}(Pr(C_{L}/\Gamma_{L}),T)\in\mathbb{Z}[T]$
and the remaining $5$ ones are determined by the symmetries of $P_{1}$. 
\begin{rem}
This algorithm for computing the action of the Frobenius on the Prym
variety $Pr(C_{L}/\Gamma_{L})$ is generalizable to other plane curves
occurring as discriminant locus of other quadric bundles, see e.g.
\cite{Beauville2}. 
\end{rem}

\section{Examples}

\label{section examples}

\subsection{Reduction in characteristic $5$ and $7$ of a cubic threefold.}

\label{Examples}

Let $F\hookrightarrow\mathbb{P}^{4}$ be the cubic threefold with
equation:
\[
F_{eq}=x_{1}x_{4}^{2}+2x_{2}x_{4}x_{5}+x_{3}x_{5}^{2}+2q_{1}x_{4}+2q_{2}x_{5}+f,
\]
 where: 
\[
\begin{array}{c}
q_{1}=x_{1}^{2}+2x_{2}^{2}+x_{2}x_{3}+x_{3}^{2}\\
q_{2}=x_{1}x_{2}+4x_{2}x_{3}+x_{3}^{2}\\
f=x_{2}^{2}x_{3}-(x_{1}^{3}+4x_{1}x_{2}^{2}+2x_{2}^{3}).
\end{array}
\]
 The cubic $F$ is smooth in characteristic $5,7,11$ and $13$ ;
it is singular in characteristic $2,3$. The associated quintic curve
$\Gamma=\Gamma_{L}$ is smooth in characteristic $3,7,11,13$, but
singular in characteristic $2,5$. 
\begin{rem}
The cubic $F_{/\mathbb{Q}}$ and its Fano surface $S_{/\mathbb{Q}}$
have bad reduction at the place $3$, however the intermediate Jacobian
$J(F)\simeq Pr(C_{L}/\Gamma_{L})$ has good reduction. This is the
same phenomena as for the curves and their Jacobian. We remark also
that the curves $\Gamma_{L}$ and $C_{L}$ both have bad reduction
at the place $5$, but the associated Prym variety has good reduction. 
\end{rem}
We have implemented the algorithm in Sage. Using a personal laptop,
it takes $5$ minutes to obtain $P_{1}(S,T)$ for $S$ over $\mathbb{F}_{7}$.

Over $\mathbb{F}_{5}$, we get (see \cite{RoulleauWorksheet}) :
\[
P_{1}(S_{/\mathbb{F}_{5}},T)=(5T^{2}+1)(625T^{8}+50T^{6}+40T^{5}-6T^{4}+8T^{3}+2T^{2}+1).
\]
 The Fano surface $S_{/\mathbb{F}_{5}}$ has Picard number $5$ and
contains $33$ $\mathbb{F}_{3}$-rational points. We have:
\[
A_{5}:=\lim_{s\to1}\frac{P_{2}(S_{/\mathbb{F}_{5}},5^{-s})}{(1-5^{1-s})^{5}}=\frac{2^{18}\cdot3^{5}\cdot157}{5^{10}}.
\]

Over $\mathbb{F}_{7}$, we get:
\[
\begin{array}{cc}
P_{1}(S_{/\mathbb{F}_{7}},T)= & 1+4T+15T^{2}+46T^{3}+159T^{4}+460T^{5}+1113T^{6}\\
 & +2254T^{7}+5145T^{8}+9604T^{9}+16807T^{10}.
\end{array}
\]
 It is an irreducible polynomial over $\mathbb{Q}$, therefore by
the Honda-Tate Theorem \cite[Theorems 2-3, App. I]{MumfordAV}, the
intermediate jacobian $J(F)$ of $F$ is simple. The Fano surface
$S_{/\mathbb{F}_{7}}$ has Picard number $5$ and $97$ $\mathbb{F}_{7}$-rational
points. We obtain:
\[
A_{7}:=\lim_{s\to1}\frac{P_{2}(S_{/\mathbb{F}_{7}},7^{-s})}{(1-7^{1-s})^{5}}=\frac{2^{4}\cdot83^{2}\cdot557\cdot5737}{7^{10}}.
\]

\begin{rem}
We prove in \cite{Roulleau} that a generic Fano surface over $\mathbb{C}$
has Picard number $\rho=1$. One would like to exhibit an example
of a Fano surface over $\mathbb{Q}$ with $\rho=1$. By reducing the
above Fano surface $S_{/\mathbb{Q}}$ modulo a prime we obtain the
bound $\rho_{S}\leq5$. Since $A_{5}/A_{7}$ is not a square in $\mathbb{Q}$,
we can apply the van Luijk method \cite{Luijk} and obtain the inequality
$\rho_{S}\leq4$.
\end{rem}
Over the field $\mathbb{F}_{11}$, the computation becomes difficult
: it needs $4$ minutes to get $N_{4}$ but more that $24$ hours
to get $N_{5}$. By \cite[Lem. 1.2.3]{Kedlaya}, for positive integers
$q,d,j$ and complex numbers $a_{1},\dots,a_{j-1}$, there exists
a certain disk of radius $\frac{d}{j}q^{j/2}$ which contains every
$a_{j}$ for which we can choose $a_{j+1},\dots,a_{d}\in\mathbb{C}$
so that the polynomial
\[
R(T)=1+\sum_{j=1}^{j=d}a_{j}T^{j}
\]
has all roots on the circle $|T|=q^{-1/2}$. In our case, $\frac{d}{j}q^{j/2}=\frac{10}{5}11^{5/2}=802.623...$
and we obtain that $P_{1}(S_{/\mathbb{F}_{11}},T)=Q_{a}$, where $a$
is an integer in $\{80,\dots,332\}$ and
\[
\begin{array}{cc}
Q_{a}= & 1-T+13T^{2}+T^{3}-28T^{4}+aT^{5}-11\cdot28T^{6}\\
 & +11^{2}T^{7}+11^{3}\cdot13T^{8}-11^{4}T^{9}+11^{5}T^{10}.
\end{array}
\]
 For all the consecutive values $a\in\{80,\dots,332\}$, the polynomial
$Q_{a}$ has its roots equal to $11^{-1/2}$ (with error at most $10^{-10}$)
and we cannot distinguish the $a$ corresponding to our Fano surface
$S$.

\subsection{The Klein cubic threefold.}

Let $S_{/\mathbb{Q}}$ be the Fano surface of lines of the Klein cubic
threefold : 
\[
F_{/\mathbb{Q}}=\{x_{1}^{2}x_{2}+x_{2}^{2}x_{3}+x_{3}^{2}x_{4}+x_{4}^{2}x_{5}+x_{5}^{2}x_{1}=0\}.
\]
 It is easy to check that $F$ (hence $S$) has good reduction at
every prime $p\not=11$.
\begin{prop}
Let us suppose $p\not=2$. If $11$ is not a square modulo $p$, then
$S_{p}$ the reduction mod $p$ is supersingular i.e. its geometric
Picard number equals $45=b_{2}$, otherwise $S_{p}$ has geometric
Picard number $25$.\end{prop}
\begin{proof}
Let be $\nu=\frac{-1+\sqrt{-11}}{2}$, $\mathcal{O}=\mathbb{Z}[\nu]$
and $E=\mathbb{C}/\mathbb{Z}[\nu]$. The intermediate jacobian $J(F)_{/\mathbb{C}}$
is isomorphic to $E^{5}$ (see \cite{RoulleauKlein}). By \cite[App. A3]{Silverman},
the elliptic curve $E$ has the following model over $\mathbb{Q}$:
\[
y^{2}+y=x^{3}-x^{2}-7x+10,
\]
 which we still denote by $E$. The curve $E$ has good reduction
for prime $p\not=11$ and it has complex multiplication by $\mathcal{O}$
(over a certain extension). We use the criteria of Deuring \cite[Chap. 13, Thm 12)]{Lang}
: for odd $p\not=11$, the reduction of $E$ modulo $p$ is a supersingular
if and only if $p$ is inert or ramified in $\mathcal{O}$. By classical
results on number theory, an odd prime $p\not=11$ is inert or ramified
in $\mathcal{O}$ if and only if $11$ is not a square modulo $p$.

Over an extension, the intermediate Jacobian $J(F)$ is isogenous
to $E^{5}$. By \cite{MumfordAV}, the geometric Picard number of
the reduction modulo $p$ of $J(F)$ is therefore $45$ if $11$ is
not a square modulo $p$, and $25$ otherwise. 
\end{proof}
\selectlanguage{french}%

\noindent Xavier Roulleau,\\
Universit\'e de Poitiers,\\
Laboratoire de Math\'ematiques et Applications, UMR 7348 du CNRS,\\
 Boulevard Pierre et Marie Curie,\\
T\'el\'eport 2 - BP 30179,\\
86962 Futuroscope Chasseneuil,\\
France\\
{\tt Xavier.Roulleau@math.univ-poitiers.fr}\\ \selectlanguage{english}%

\end{document}